\documentclass{amsart}

\usepackage{tikz}
\usepackage{color,soul}

\title[]{Wreath product in automorphism groups of graphs}
\author{Mariusz Grech, Andrzej Kisielewicz}
\address{Institute of Mathematics, University of Wroclaw \\ 
pl.Grunwaldzki 2, 50-384 Wroclaw, Poland}
\email{Mariusz.Grech@math.uni.wroc.pl}
\thanks{Supported in part by Polish NCN grant 2016/21/B/ST1/03079}

\keywords{Colored graph, automorphism group, permutation group, wreath product.}

\begin{document}

\newtheorem{Theorem}{Theorem}[section]
\newtheorem{Cor}[Theorem]{Corollary}
\newtheorem{Lemma}[Theorem]{Lemma}
\newtheorem{Fact}[Theorem]{Fact}
\newtheorem{Proposition}[Theorem]{Proposition}
\newtheorem{Corollary}[Theorem]{Corollary}

\newcommand{\row}[2]{I_{#2}\!\times\! {#1}}
\newcommand{\Wr} {\mbox{\hspace{0.3ex}$\wr$\hspace{-0.2ex}$\wr$\hspace{0.3ex}}}

\newcommand{\clo}[1]{\overline{#1}}
\newcommand{\cld}[1]{\overleftarrow{#1}}


\begin{abstract}
The automorphism group of the composition of graphs $G\circ H$ contains the wreath product $Aut(H)\wr Aut(G)$
of the automorphism groups of the corresponding graphs. The classical problem considered by Sabidussi and Hemminger was under what conditions $G\circ H$ has no other automorphisms. In this paper we deal with the converse. If the automorphism group of a graph (or a colored graph or digraph) is the wreath product $A\wr B$ of permutation groups, then the graph must be the result of the corresponding construction. The question we consider is whether $A$ and $B$ must be the automorphism groups of graphs involved in the construction. We solve this problem, generally in positive, for the wreath product in its natural imprimitive action (which refers to the results by Sabidussi and Hemminger). Yet, we consider also the same problems for the wreath product in its product action, which turns out to be more complicated and leads to interesting open questions involving other combinatorial structures.
\end{abstract}

\maketitle

Wreath product is one of the most important combinatorial constructions in the theory of groups and permutation groups. It is enough to mention its crucial role in the O'Nan-Scot theorem and the classification of the maximal subgroups of $S_n$. Also, many important configurations, as the Hamming scheme $H(n,m)$, involve the wreath product of permutation groups. The corresponding structure in the theory of graphs is the \emph{composition} of graphs, called also the \emph{lexicographic product} or more recently the \emph{wreath product} of graphs (see \cite{ara,DM}).  

The automorphism group $Aut(G\circ H)$ of the composition of graphs $G$ and $H$ contains the wreath product $Aut(H)\wr Aut(G)$ of the automorphism groups of the graphs $H$ and $G$, and often the equality $Aut(G\circ H) = Aut(H)\wr Aut(G)$ holds. 
The classical results by Sabidussi \cite{sa1,sa2} and Hemminger \cite{he1,he2,he3} give the necessary conditions for the equality above to hold, or using Hemminger words, for the above product of graphs to have no ``unnatural'' automorphisms. This was  generalized by Hahn \cite{ha1,ha2} for directed graphs and hypergraphs, and by Dobson and Morris \cite{DM} for colored graphs. As a matter of fact, colored graphs seem a more adequate setting for this kind of problems (see e.g., \cite{gre1,GK1,sib}). This is due to the fact that
the fundamental construction connected with the question whether a permutation group $A$ is the automorphism group of a graph is
the colored graph $G^*(A)$ defined by orbitals of the group. Note that automorphisms of colored graphs are also considered in the literature in terms of \emph{permutations preserving factorizations} of complete graphs (see e.g., \cite{GLPP,LLP}).   

In general, necessary conditions for $Aut(G\circ H) = Aut(H)\wr Aut(G)$ are technical and not easy to state, but in case of finite graphs and digraphs they are quite simple. In particular,  for finite colored graphs $H_1$ and $H_2$,  $Aut(H_1\circ H_2) =$ \hbox{$Aut(H_2)\wr Aut(H_1)$} if and only if, for every color $k$, the following implication holds: if $H_1$ has a pair of $k$-twins (two vertices joined by an edge of color $k$ that have exactly the same neighbors in every color), then the $k$-complement of $H_2$ is connected. Even if the condition is not satisfied one can construct a graph $G$ such that $Aut(G) = Aut(H_2)\wr Aut(H_1)$. In \cite{je}, we have proved\footnote{The results of \cite{ha1,je} have been rediscovered in \cite{DM} without citing.} that if $A,B$ are automorphism groups of some $r$-colored graphs, than there exists an $r$-colored graph $G$ with $Aut(G) = A\wr B$. 

In this paper we consider the converse of these results. Suppose that the automorphism group of a graph $G$ has the form of the wreath product $A\wr B$ of two permutation groups $A$ and $B$. Does it mean that $A$ and $B$ are the automorphism groups of some graphs $H_1$ and $H_2$, and if so, can $G$ be obtained as the composition of $H_1$ and $H_2$?

We show that this is true for vertex transitive colored graphs and digraphs, but for intransitive graphs and digraphs the situation is more complicated and involves a generalization of composition of graphs. All these results concern the wreath product in its \emph{imprimitive action}. In the second part of the paper we consider the same problem for the wreath product in the \emph{product action}. Here the situation turns out to be more complex and more intriguing. Partial characterizations we obtain involve the automorphism groups of colored hypergraphs. The open question we state is strictly connected with the question of characterizing permutation groups that cannot be represented as the automorphism group of a colored hypergraph. Are there such groups other than the alternating groups and a few know exceptional examples? We address this question in more detail at the end of the paper.


\section{Definitions and basic facts}\label{sec1}
Our terminology on graphs and permutation groups is standard. In this section we recall the most important definitions, fix notation, and remind some basic facts.

By a \emph{colored graph} $G$ (or more precisely, \emph{$k$-colored graph}), we mean a pair $G = (V,E)$, where $V$ is the set of vertices of $G$, and $E$ the \emph{color function} from the set $P_2(V)$ of unordered pairs of the set vertices $V$ into the set of colors $\{ 0, \ldots, k-1\}$. Thus, $G$ is a complete graph with colored edges. It can be also viewed as an arbitrary colored graph where color $0$ stands for non-edges. 

An automorphism of a colored graph $G=(V,E)$ is a permutation $\sigma$ of the set $V$
preserving the edge function: $E\{ v, w\} = E \{ v\sigma, w\sigma \}$, for all $v,w \in V$. 
(Note that, in this paper, we adopt the convention to write permutations on the right. Also, for visibility, we omit a pair of parentheses in denoting the color function, and write $E\{v,w\}$ rather than  $E(\{v,w\})$). 
The group of automorphisms of $G$ will be denoted by
$Aut(G)$, and considered as a permutation group $(Aut(G),V)$ acting on the set of the vertices $V$. 
Notions and notations for directed graphs are analogous. 

Permutation groups are treated up to permutation isomorphism. Generally, a
permutation group $A$ on a set $V$ is denoted $(A,V)$ or just $A$, if the
set $V$ is clear from the context. It is always assumed that for a permutation group the cardinality $|V|>1$. Similarly, colored graphs are treated up to color isomorphism (two colored graphs are color isomorphic if they are isomorphic after suitable renaming colors of the edges). 

We will consider the class $GR$ of all permutation groups that are automorphism groups of colored graphs, and the class $DGR$ of all those permutation groups that are automorphism groups of colored directed graphs.

For finite graphs, these classes have been introduced, in fact, by H. Wielandt in \cite{wie},
where permutation groups that are automorphism groups of (systems of relations corresponding to) colored digraphs
were called $2$-closed, and those that are automorphism groups of colored graphs where referred to as $2^*$-closed.

In~\cite{je}, it was proved, in particular, that if the permutation groups $A$ and $B$ belong both to $GR$, then the wreath product $A \wr B$ of this groups also belongs to $GR$. The result holds for the wreath product both in the imprimitive action and the product action. 
It has been proved for finite graphs in a more detailed setting taking into account the number of colors involved. But for the general version formulated above the proof easily generalizes, also for the corresponding versions for the classes $DGR$. For the case of imprimitive action, these facts were proved independently in \cite{DM}. In this paper we formulate and prove the converse. 

Given a permutation group $(A,V)$, we shall consider the natural actions of $A$
a on the sets of ordered and unordered pairs of $V$, denoted $V\times V$ and $P_2(V)$, respectively. 
The orbits of $A$ in the action on $V\times V$ are called \emph{orbitals} of $A$, or more precisely, $2$-\emph{orbitals}. The $2$-orbitals consisting of pairs of the form $(v,v)$ are called trivial.  
For two orbitals $O_1, O_2$ we say that $O_1$ is \emph{paired}  with
$O_2$ if and only if $O_2 = \{(w,v):  (v,w) \in O_1\}$.
We call an orbital $O$ \emph{self-paired} if it is paired with itself. 
Moreover, we say that a permutation $\alpha$ \emph{transposes} $O_1$ and $O_2$, if $O_1\alpha
= O_2$.

The orbits of $A$ in the action on $P_2(V)$ will be called here \emph{$2^*$-orbitals}. 
Note that we can think of a $2^*$-orbital as a pair of nontrivial paired orbitals (or a nontrivial self-paired orbital). 

Let $O_0, \ldots O_{k-1}$ be all the $2^{*}$-orbitals of a permutation group $(A,V)$. We define a colored graph $G^*(A) =(V,E)$, where 
$$
E\{ v,w \}= i \textrm{\,\, if and only if the edge } \{ v,w\} \textrm{ belongs to the orbit } O_i.
$$
This graph will be called the (\emph{colored}) \emph{orbital graph} of $A$.
We observe that for the classes of permutation groups on a set $V$ and classes of colored graphs on $V$ (considered up to suitable isomorphisms), the operators $G^*$  and $Aut$ form a (monotone) \emph{Galois connection}. In particular, the composition of this two operations yields a closure operator for permutation groups, which we denote $\clo{A}=Aut(G^*(A))$. 

Thus, we have $ A \subseteq \clo{A}$, and  $A\subseteq B$ if and only if  $\clo{A} \subseteq \clo{B}$. Moreover,  $A \in GR$ if and only if $A = \clo{A}$ 
(see \cite{gre1}, where theses facts are proved without referring to Galois connection). 

The dual closure operator for colored graphs is $\clo{G} = G^*(Aut(G))$. Here the corresponding order relation is the \emph{subcoloring}. A colored graph $H$ is a \emph{subcoloring} of $G$ if $H$ can be obtained from $G$ by partitioning some colors into larger number of new colors. We write then $H \preceq G$. Here, we have $ \clo{G} \preceq G$ (unlike the previous case), and  $H\preceq G$ if and only if  $\clo{H} \preceq \clo{G}$. Moreover,  if $A = Aut(G)$, then, $G^*(A) \preceq G$. 

Replacing $2^{*}$-orbitals by $2$-orbitals we define analogously the colored \emph{orbital digraph} $G(A)$. Then operators $Aut$ and $G(\dots)$ form a Galois connection between classes of permutation groups on a set $V$ and classes of colored digraphs on $V$. We get another closure operator $\cld{A} = Aut(G(A))$, with the same mentioned properties.  In addition, we have 
$A\subseteq \cld{A} \subseteq \clo{A}$. Also, every digraph $D$ that has the automorphism group $A$ is a subcoloring of $G(A)$.

By $S_V$ we denote the group of
all permutations on the set $V$. If $V$ is a set consisting of $n$ elements, then we write also $S_n$ for $S_V$. By $I_n$ we denote the least subgroup of $S_n$, i.e., that consisting of the identity permutation only. 
The identity permutation is denoted generally by $id$, if no confusion can arise.

Given permutation groups $A, B$ acting on sets $V$ and $W$, respectively, by
the {\it direct product} $A \times B$  we mean the permutation group acting on the Cartesian product $V \times W$ consisting of all permutations  $\gamma =(\alpha,\beta)$ with $\alpha \in A$,
$\beta \in B$, such that 
$$(v,w)(\alpha,\beta) = (v\alpha,w\beta)$$ 
for every $(v,w) \in V\times W$. 

A special role in this paper plays the permutation group $I_n \times A = A \times I_n$ (up to permutation isomorphism). This group acts just as $(A,V)$, but the action is on $n$ copies of $V$ and is done in a parallel way. Therefore, it is called a \emph{parallel multiple} of $A$ (note that in \cite{kis1}, and some other papers, it is called the \emph{parallel power}).

\section{The imprimitive action of the wreath product}\label{imp}

In this section, we study the wreath product of permutation groups in its natural imprimitive action. 
For permutation groups $(A, V)$ and  $(B,W)$, by the (imprimitive) \emph{wreath} product $A$ and $B$, denoted $A \wr B$, we mean a permutation group acting on $V \times W$, consisting of all permutations $\gamma$ for which there exist permutations $\alpha_w \in A$ for each $w\in W$ and permutation $\beta\in B$ such that 
\begin{equation} 
 (v,w)\gamma = (v\alpha_w,w\beta) 
\end{equation}
for every $(v,w)\in V\times W$.

Abstractly, the wreath product is a semidirect product of $A^{B}$ and $B$; for this fact and more general definition see \cite{cam}. In this paper, we define $A \wr B$ merely as a permutation group (a set of permutations) abstracting from its action on $V\times W$, because our problem does not depend on this action. This approach makes the definition easier and simplifies the proofs.

Following Cameron~\cite{cam}, we may think of the wreath product $A \wr B$ as a fibre bundle over the set $W$. In pictures (like the left hand side of Figure~1), fibres are presented as parallel vertical lines over the horizontal line representing the set $W$. Each fibre is a copy of $V$ with points permuted by elements of $A$, independently in each fibre, and whole fibres are permuted using permutations of $B$. Thus fibres form natural blocks of $A \wr B$.

The corresponding construction for graphs is the \emph{composition} of graphs. For two colored graphs, the composition $G \circ H$ of graphs $G$ and $H$ with disjoint vertex sets $W$ and $V$ is the graph with the vertex set $W\times V$ and the edges colored as follows:
the edge from  $\{w_1,v_1\}$ to $\{w_2,v_2\}$ has color $c$, if the edge $\{w_1,w_2\}$ forms an edge in $G$ of color $c$ or $w_1=w_2$ and the edge $\{v_1,v_2\}$ has color $c$ in $H$. (For directed graphs the definition is the same). 

If $G$ and $H$ are \emph{color-disjoint} (colored with disjoint sets of colors), then for the automorphism groups, considered as permutation groups, we have $Aut(G \circ H) = Aut(H) \wr Aut(G)$ (note the reverse order in the notation). Most often this equality holds also if $G$ and $H$ are not color-disjoint (see~\cite{DM}). Therefore, in cite~\cite{ara,DM}, this construction is called the \emph{wreath product of graphs}. 

The reader should be warned that in many papers, as  in~\cite{DM}, the opposite order of components in notation of the wreath product of groups is used  (which may be a source of misunderstanding). We follow the tradition applied in Cameron~\cite{cam}.

In order to describe the structure of the graph $G^*(A\wr B)$ we introduce the free composition of colored graphs. \medskip

\noindent \textbf{Definition.} Let $H=(V,E_1)$ and $G=(W,E_2)$ be $r$-colored and $s$-colored graphs, respectively. Let $O_1,\dots,O_t$ be the orbits of $Aut(H)$ and $Q_1,\dots,Q_u$ the orbits of $Aut(G)$. The \emph{free composition} of $G$ and $H$, denoted $G \circ_{f} H $, is the colored graph $(W\times V, E)$ with the color function $E$ is defined as follows:
$$ 
E\{(w_1,v_1),(w_2,v_2)\}= \left\{ 
\begin{array}{ll}
(i,j,d)  & \textrm{if } 
w_1 \neq w_2,  E_2\{w_1,w_2\}=d, v_1 \in O_{i} \textrm{ and} \\ & v_2 \in O_{j}, \\
(k,c)  &  \textrm{if } w_1=w_2,   E_1\{v_1,v_2\}=c,  \textrm{ and } w \in Q_{k}.
\end{array} 
\right.
$$ \medskip

Here, the colors of $H\circ_{f}G$ are labelled by the pairs
$(k,c)$ and the triples $(i,j,d)$, with $k\leq r$, $i,j\leq s$, and $d$ and $c$ being colors of $G$ and $H$, respectively. More formally, they should be replaced by suitable numbers from $0$ to $r^2u+st-1$.   (The definition for colored directed digraphs is analogous). 

The situation is especially clear under assumption
that $G$ and $H$ are color-disjoint. In this case, if $Aut(G)$ and $Aut(H)$ are transitive, then $G\circ_f H$ is simply the composition $G\circ H$ of graphs. If $Aut(H)$ is transitive, and $Aut(G)$ intransitive, then $G\circ_f H$ is what is called the $C$-join in \cite{DM} (or $X$-join in \cite{he3}). 

It is also easy to describe the graph $G^*(A\wr B)$ in terms of the free composition of colored graphs. Looking at the orbitals in $A\wr B$ we see easily that if groups $A$ and $B$ are transitive, then the graph $G=G^*(A \wr B)$ is just the composition of  (color-disjoint) $G^*(B)$ and $G^*(A)$. In general, the edges joining points in the fibre of $A\wr B$ (called \emph{vertical edges}) form the graph $G^*(A)$, while the remaining edges (\emph{non-vertical edges}) are determined by $G^*(B)$ according to the following. 

\begin{Proposition}\label{p:1}
$G^*(A\wr B) = G^*(B) \circ_f 
G^*(A)$
\end{Proposition}

The construction is presented schematically at the left hand side of Figure~1. After these remarks we are ready to prove the following.

\begin{Theorem}\label{th:1}
If an $r$-colored graph $G$ is vertex transitive, and $Aut(G) = A\wr B$, for some permutation groups $A$ and $B$, then there exist $r$-colored graphs $H_1$ and $H_2$ such that $Aut(H_1)=A$ and $Aut(H_2)=B$, and $G=H_2 \circ H_1$.
\end{Theorem}

\begin{proof}
Since $Aut(G) = A\wr B$, graph $G$ is an $r$-colored subcoloring of $G^*(A\wr B)$. Since $Aut(G)$ is transitive, both $A$ and $B$ are transitive.  Consequently, $G^*(A\wr B) = G^*(B) \circ 
G^*(A)$. 

It follows that $G = H_2 \circ H_1$, where $H_2, H_1$ are $r$-colored subcolorings of $G^*(B)$ and $G^*(A)$, respectively. 
This implies that $Aut(H_2) \supseteq Aut(G^*(B)) = \clo{B}$, and similarly,  $Aut(H_1) \supseteq \clo{A}$.

Thus, we have  $A \wr B = Aut(G) \supseteq Aut(H_1) \wr Aut(H_2) \supseteq \clo{A} \wr \clo{B}$. It follows that $A=Aut(H_1) =\clo{A}$ and $B=Aut(H_2) =\clo{B}$,  proving the result.
\end{proof}

The analogous results hold for colored digraphs, as well. If $G$ is not vertex transitive (which means that at least one of $A$ or $B$ is not transitive), then the situation is more complex. First of all, if $Aut(G) = A\wr B$, then, as we shall see, $A$
and $B$ may not belong to $GR$ at all. Yet, even if $A,B\in GR$, then although $G$ is a subcoloring of $G^*(A\wr B)$, it needs not to be the compositions of two suitable graphs, since now one can use edges between points with coordinates in different orbits to block unnecessary automorphisms. In fact, this possibility is used in the proof of Theorem~3.1 in \cite{je}.  
A part of this theorem may be generalized as follows.

\begin{Theorem}\label{reprezentowalne}
Given two permutation groups $A$ and $B$, the following hold

\begin{enumerate}
\item If $A, B \in GR$, then $A \wr B \in GR$. 
\item If $A, B \in DGR$, then $A \wr B \in DGR$.
\end{enumerate}
\end{Theorem}

As we have already pointed out earlier, in \cite{je}, the above was proved only for finite graphs and in more detailed form involving the numbers of colors. Yet, if we ignore calculations concerning the numbers of colors, then one may obtain easily the proofs of the above statements, working also for infinite graphs and digraphs. We will need also the following simple observation from \cite{gre1}:

\begin{Lemma}\label{l:orbits}
If $A\neq I_2$, and a permutation $\alpha$ preserves the $2^*$-orbitals of $A$, then $\alpha$ preserves the orbits of $A$
\end{Lemma}

It follows easily from the fact that edges with endvertices in a nontrivial orbit form a union of $2^*$-orbitals (for details see~\cite[Lemma~2.2]{gre1}).

In the rest of this section, we consider the wreath products of the permutation groups in which at least one of the components does not belong to the class $GR$ (or $DGR$). The problem is more difficult for undirected graphs, so we focus on undirected graphs.

In the lemmas below we assume generally that $A$ and $B$ are permutation groups acting on sets $V$ and $W$, respectively.

\begin{Lemma}\label{B}
If $A \not\in GR \cup \{ I_2\}$, then, $A \wr B \not\in GR$. 
\end{Lemma}

\begin{proof}
Let $G=G^*(A \wr B)$. By the properties of the closure discussed in Section~\ref{sec1}, it is enough to show that there exists $\phi\in Aut(G)\setminus  (A \wr B)$.

Since $A \notin GR$, there is $\alpha \in \clo{A}\setminus A$. We use it to define a permutation $\phi$ on $V\times W$. For a fixed element $w_0\in W$ we put
$$(v,w)\phi=\left\{
\begin{array}{ccc}
(v\alpha,w_0) & \textrm{for} & w=w_0\\
 (v,w) & \textrm{for} & w\ne w_0
\end{array}\right.$$

We will show that $\phi$ preserves the colors in $G$. 
Let $e=\{x,y\}$ be an edge of $G$. 
If none of $x,y$ belongs to $V \times \{w_0\}$, then $\phi(e)=e$, and the claim is obvious. Also, if both $x$ and $y$ belong to $V \times \{w_0\}$, then the colors of $\phi(e)$ and $e$ are the same (since the colored graph spanned on $V \times \{V_0\}$ is a copy of $G^*(A)$ and $\alpha \in \clo{A}$).

Assume that $x \in V \times \{w_0\}$, $y \notin V \times \{w_0\}$. Then, $y\phi=y$, and for $x=(v,w_0)$, $x\phi=(v\alpha,w_0)$. By Lemma~\ref{l:orbits}, $v$ and $v\alpha$ belong to the same orbit of $A$. 
Consequently, there exists $\alpha' \in A$ such that $v\alpha' = v\alpha$. 
Define $\phi'$ as $\phi$ above with $\alpha$ replaced by $\alpha'$.
Then, $\phi' \in A \wr B$, and 
$x\phi' = x\phi$ and $y\phi' = y = y\phi$. Therefore  the edges $e$ and $e\phi=e\phi'$ have the same color, completing the proof.
\end{proof}

Now, we show that in the remaining cases, whether $A \wr B \in GR$ or not depends on the number $t$ of orbits of $A$ and on whether the parallel multiple $\row{B}{t} \in GR$ or not.

\begin{Lemma} \label{A}
If $t\geq 1$ is the number of orbits of $A$, and   $\row{B}{t} \not\in GR$ and $B\neq I_2$, then  $A \wr B \not\in GR$. 
\end{Lemma}

\begin{proof} 
Let $O_1, \ldots, O_t$ be the orbits of $A$. 
Let $G=G^*(A \wr B)$. Then, $G = (V\times W,E)$ for some edge function $E:P_2(V\times W)\to \{0, \ldots, k-1\}$. 
Assume to the contrary that $A \wr B \in GR$. 
Then, $Aut(G)=A \wr B$. 
We use this fact to construct a graph $G'=(W\times\{1, \ldots, t\},E')$ such that $Aut(G')=\row{B}{t}$, thus contradicting the assumption that  $\row{B}{t} \not\in GR$.  

Let $Q_1,\ldots, Q_s$ be the orbits of $B$, and denote $F(r,i,j) = rt^2+it+j$.
We define the edge function $E'$ as follows.

$$ E'\{(w_1,i),(w_2,j)\}= \left\{ 
\begin{array}{ll}
r  & \textrm{if } 
w_1 \neq w_2, \textrm{ and for some } v_1 \in O_{i} \textrm{ and} \\ & v_2 \in O_{j}, \ E\{ (v_1,w_1), (v_2,w_2)\}=r,\\
F(r,i,j)+k  &  \textrm{if } w_1=w_2, \textrm{ and } w_1\in    Q_r.
\end{array} 
\right.$$ 

First observe that this definition does not depend on the choice of $v_1$ and $v_2$ in the first line. Indeed, this is so, since for fixed $w_1$ and $w_2$, the colors of edges $\{(v_1,w_1), (v_2,w_2)\}$ in $G$ are determined by orbits to which $v_1$ and $v_2$ belong. Thus, $E'$ is well-defined. 

\begin{figure}\label{fig1}
\begin{tikzpicture}[auto,inner sep=1pt, 
minimum size=5pt]

  
      \draw[-] (1,0)--(5,0);

         \draw[-] (1,0)--(1,4);
        \draw[-] (2,0)--(2,4);
          \draw[-] (3,0)--(3,4);
        \draw[-] (4,0)--(4,4);
            \draw[-] (5,0)--(5,4);
            

\foreach \b in {0,1,2,3}{    
  \foreach \a in {0,1,...,4}{            
\draw (\a+7,\b+0.5) node[circle,fill=black,draw]{};            }}

\draw[-] (7,0.5)--(11,0.5);
\draw[-] (7,1.5)--(11,1.5);
\draw[-] (7,2.5)--(11,2.5);
\draw[-] (7,3.5)--(11,3.5);
      
\foreach \b in {0,1,2,3}{    
  \foreach \a in {0,1,...,4}{ 
  \draw  (\a+1,\b+0.5) ellipse (1mm and 5mm);}
  }
  
\draw[-,thick] (1.2,2.7)--(2.8,3.7);
\draw[-,thick] (1.2,2.3)--(2.8,3.3);
\draw[-,thick] (7,2.5)--(9,3.5);

\draw[-,thick] (4.2,2.7)--(4.8,3.7);
\draw[-,thick] (4.2,2.3)--(4.8,3.3);
\draw[-,thick] (10,2.5)--(11,3.5);

\draw[-,thick,dashed] (2.2,1.6)--(4.8,0.6);
\draw[-,thick,dashed] (2.2,1.4)--(4.8,0.4);
\draw[-,thick,dashed] (8,1.5)--(11,0.5);

\draw[-,thick,dotted] (3.2,2.7)--(3.8,1.7);
\draw[-,thick,dotted] (3.2,2.3)--(3.8,1.3);
\draw[-,thick,dotted] (9,2.5)--(10,1.5);

\draw (0.2,2.5) node {$G^*(A)$};
\draw (0.4,1.5) node {$V$};
\draw (6.2,2) node {$t$};
\draw (3,-0.5) node {$G^*(B); \;\; W$};
\draw (9,-0.5) node {$G^*(B); \;\; W$};
\draw (3,4.5) node {$G=G^*(A\wr B)$};
\draw (9,4.5) node {$G' = G^*(B\times I_t)$};

 

  \end{tikzpicture}
\caption{Corresponding non-vertical edges in $G$ and $G'$}
\end{figure}
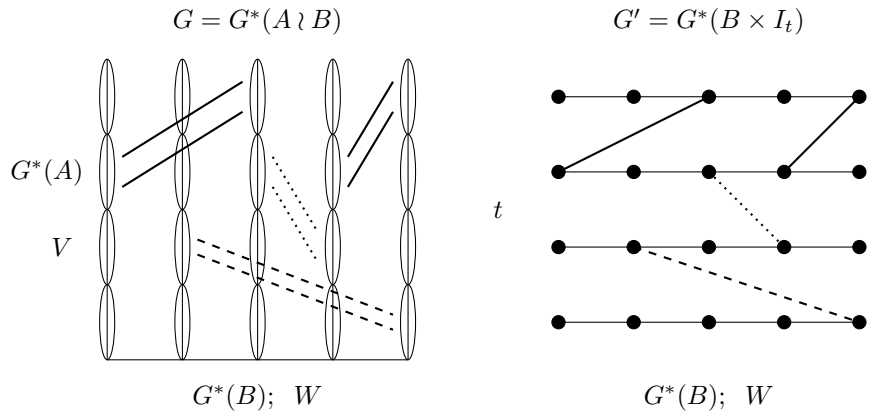

The construction is presented schematically in  Figure~1. The graph $G=$ $G^*(A \wr B)$ on the left was discussed earlier. 
The graph $G'$, on the right, is presented as $t$ copies of $(B,W)$ put horizontally. The figure shows that
the points on vertical lines in $G'$ correspond to the orbits on vertical lines in $G$. The colors of the vertical edges  in $G'$  are given by the second line of the formula for $E'$. 

The colors of non-vertical edges in $G'$ are determined by the colors of the corresponding non-vertical edges in $G$.
We make use of the fact that two orbits $O$ and $Q$ in various copies of $A$ in $G$ all the edges joining $O$ and $Q$ have the same color. 
In other words, the non-vertical edges in $G'$ are obtained from the corresponding non-vertical edges in $G$ just by collapsing orbits of $A$ to points.

Note that the value $F(r,i,j)+k$ is always greater than the value $r$ in the first line of the definition of $E'$, which means that the colors of vertical edges are always different than the colors of non-vertical edges. Moreover, since for different triples $(i,j,r)$ the values of $F(r,i,j)+k$ are different, it means that different vertical edges in $G'$, say  $\{(w_1,i_1), (w_1,j_1)\}$ and $\{(w_2,i_2), (w_2,j_2)\}$,  have  different colors unless $w_1$ and $w_2$ are in the same orbit of $B$, $i_1=i_2$, and $j_1=j_2$.

We show that $\row{B}{t} = Aut(G')$.
Let $\phi \in \row{B}{t}$. We  show that $\phi$ preserves the edge colors in $G'$. Indeed, we have $(w,i)\phi = (w\beta,i)$ for some $\beta\in B$ (for any any $w$ and $i$). If $w_1\neq w_2$, then $w_1\beta \neq w_2\beta$, and
 $$\{(w_1,i),(w_2,j)\}\phi = \{(w_1\beta,i),(w_2\beta,j)\},$$ 
 which by definition of $E'$ has the same color $r$ as $\{(w_1,i),(w_2,j)\}$. In turn, 
 $$\{(w,i),(w,j)\}\phi = \{(w\beta,i),(w\beta,j)\},$$ has the same color as $\{(w,i),(w,j)\}$,
 since $w\beta$ is in the same orbit as $w$. 
 Consequently, $\row{B}{t} \subseteq Aut(G')$.

For the opposite inclusion, assume now that $\phi \in Aut(G')$. By the remarks following the definition of $E'$ (and since $B\neq I_2$), $\phi$ preserves the partition into sets of the form $\{(w,i) : i=1,\ldots,t\}$ 
(columns of $G'$) and fixes each set of the form $\{(w,i) : w\in B\}$ (rows of $G'$).  This means that $\phi \in S_W\times I_t$, and is of the form $(w,i)\phi = (w\gamma,i)$, for some permutation $\gamma$ of $W$. Since $\phi$
preserves the colors of non-vertical edges in $G'$ (that by the definition of $E'$, are obtained by collapsing orbits in $A$ in $G$ to single points), the permutation $(v,w)\phi' = (v,w\gamma)$
preserves the colors of edges in $G$. Hence, $\phi' \in A \wr B$, and consequently $\gamma\in B$. This proves $\row{B}{t} \supseteq Aut(G')$, and completes the proof. \end{proof}

\begin{Lemma}\label{C}
Let $A\in GR \cup \{I_2\}$. If $A$ has $t\geq 1$ orbits,
and  $\row{B}{t} \in GR$, then $A \wr B \in GR$. 
\end{Lemma}

\begin{proof} 
For $t=1$ the result follows by  Theorem~\ref{reprezentowalne}.
If $A=I_2$, then $A \wr B = I_2 \wr B = \row{B}{2}$ and the statement is trivial. 
In the remaining part of the proof, we assume that $A \in GR$ and $t\geq 2$. 

The proof is in a sense the converse of the previous proof. Now, we construct a colored graph $G = (V\times W,E)$ such that 
$Aut(G) = A \wr B$ using a graph $G'$ such that $Aut(G')= \row{B}{t} $ and a graph $G''$ such that $Aut(G'')= A$. For non-vertical edges in $G$ we take colors of the corresponding non-vertical in $G'$ (as in Figure~1). For vertical edges, the vertical lines in $G$ are just the copies of $G''$, yet we take copies with different coloring for different orbits of $B$. The formal definition is as follows. 

Let $E'$ and $E''$ be color functions  for graphs $G'$  and $G''$, respectively. Let $k$ be the maximum of the number of colors used in $G'$ and in $G''$.

Let $O_1, \ldots, O_t$ be the orbits of $A$. 

We define a graph $G$ such that $Aut(G) = A \wr B$. 

$$ E\{(v_1,w_1),(v_2,w_2)\}= \left\{ 
\begin{array}{ll}
E'\{(w_1,i),(w_2,j)\} & \textrm{if } 
w_1 \neq w_2, \textrm{ and } v_1 \in O_{i} \\ &\textrm{and }  v_2 \in O_{j}, \\
E''\{ v_1,v_2\} + kr &  \textrm{if } w_1=w_2, \textrm{ and } w_2\in    Q_r.
\end{array} 
\right.$$

We show that  $A \wr B = Aut(G)$.   
First suppose that $\phi\in Aut(G)$. Since the colors of vertical edges are different from the colors of the remaining edges, $f$ preserves the partition into vertical lines. It follows that $(v,w)\phi = (v\alpha_w,w\beta)$ for some permutations $\alpha_w$ of $V$ and some permutation $\beta$ of $W$. We need to show that $\alpha_w\in A$ and $\beta\in B$.

Since, for a fixed $w$, $\phi$ preserve the colors of the edges of the form $\{(v_1,w),(v_2,w)\}$, it follows that 
$$E\{(v_1,w),(v_2,w)\}=
E\{(v_1\alpha_w,w\beta),(v_2\alpha_w,w\beta)\}.$$
Consequently, $E''(v_1,v_2) = E''(v_1\alpha_w,v_2\alpha_w)$, which means that $\alpha_w\in A$. The same argument for non-vertical edges yields 
$$E\{(v_1,w_1),(v_2,w_2)\}=
E\{(v_1\alpha_{w_1},w_1\beta),(v_2\alpha_{w_2},w_2\beta)\}.$$
Since, $w_1\beta \neq w_2\beta$, and  $v_1\alpha_{w_1}$ is in the same orbit $O_i$ of $A$ as $v_1$, while $v_2\alpha_{w_2}$ is in the same orbit $O_j$ of $A$ as $v_2$, we get
$$E'\{(w_1,i),(w_2,j)\} = E'\{(w_1\beta,i),(w_2\beta,j)\},$$
for all $i,j$. It follows that the permutation $\psi$ on the set $W\times \{1,\ldots,t\}$ given by $(w,i)\psi = (w\beta,i)$ preserves the colors of $G'$. Hence, $\beta\in B$. This proves that $\phi\in A\wr B$, and consequently,  $Aut(G) \subseteq A \wr B$. 

For the converse, assume that $f\in A\wr B$. It means that $(v,w)\phi = (v\alpha_w,w\beta)$ for some  $\alpha_w \in A$ and $\beta\in B$. We show that $\phi$ preserves the colors of the edges in $G$. First consider the vertical edges, i.e. those of the form $e =\{(v_1,w),(v_2,w)\}$. It follows that 
$$E(e\phi)=
E\{(v_1\alpha_w,w\beta),(v_2\alpha_w,w\beta)\} = E''\{v_1\alpha_w,v_2\alpha_w\} +kr,$$ 
where $O_r$ is the orbit of $B$ containing $w\beta$ and $w$. Since $E''\{v_1\alpha_w,v_2\alpha_w\} = E''\{v_1,v_2\}$, we infer that $E(e\phi)=E(e)$, as required.
For non-vertical edges $e = \{(v_1,w_1),(v_2,w_2)\}$ with $w_1\neq w_2$ (and with $v_1\in O_i$ and $v_2\in O_j$) we have 
$$E(e\phi)=
E\{(v_1\alpha_{w_1},w_1\beta),(v_2\alpha_{w_2},w_2\beta)\} = $$ $$ E'\{(w_1\beta,i),(w_2\beta,j)\} = E'\{(w_1,i),(w_2,j)\} = E(e),$$
which shows that $\phi$ preserves the colors of $G$. It follows that $Aut(G) \supseteq A \wr B$, completing the proof. 
\end{proof}

To summarize the results of this section, we make use of the following nice characterization 

\begin{Proposition}\label{R} \cite[Corollary 3.7]{gre1}
For any permutation group $B$ and $t > 1$,  we have $\row{B}{t} \in GR$ if and only if $B \in DGR$. 
\end{Proposition} 

Combining the above with the lemmas proved in this section we get the following.

\begin{Theorem} \label{kon}
Let $A$ and $B$ be permutation groups.
Then, $A \wr B \in GR$ if and only if $A \in GR \cup \{I_2\}$ and one of the following holds. 
\begin{enumerate}
    \item $B \in GR\cup \{I_2\}$, or 
\item $B \in DGR \setminus (GR\cup \{I_2\})$ and $A$ is intransitive. 
\end{enumerate}
\end{Theorem}
\begin{proof}
The ``if'' part follows directly form Lemmas~\ref{C} combined with Proposition~\ref{R}. For the ``only if'' part we use Lemmas~\ref{B} and~\ref{A}. The only  unsettled case by these lemmas is $A\wr I_2$. Here we apply \cite[Lemma~3.1]{grekis1} and the fact that   $A\wr I_2 = A \oplus A$. This yields that   $A\wr I_2 \in  GR$ if and only if $A \in GR \cup \{I_2\}$, and completes the proof of the theorem. 
\end{proof} 

All the proofs in this section may be easily modified to the case of colored directed graphs. The situation is in fact simpler, and some parts of the proof are void. As a result we obtain the following.

\begin{Theorem} \label{kondir}
Let $A$ and $B$ be permutation groups. 
Then $A \wr B$ $\in DGR$ if and only if both $A, B\in DGR$. 
\end{Theorem}

\section{Product action of the wreath product}\label{pro}

Now we study the action of the wreath product on the set $V^W$ of functions from a set $W$ to a set $V$, refereed to as the product action. We consider the pairs $(\beta, (\alpha_w)_{w\in W})$, where $\beta\in S_W$ and $(\alpha_w)_{w\in W}$ is a family of permutations in $S_V$ indexed by elements of $W$. For any such pair we define a permutation $\phi$ of the set $V^W$ given by the formula
\begin{equation}
(f\phi)(w) = (f(w\beta))\alpha_w,  \label{eq:vw}  \end{equation} 
for any $f\in V^W$ and any $w\in W$.

Given two permutation groups $(A,V)$ and $(B,W)$, 
by  $(A \Wr B)$ we denote the permutation group consisting of all permutations $\phi$ of the set $V^W$ defined by the formula $(\ref{eq:vw})$ with  $\beta\in B$ and $\alpha_w \in A$ for all $w \in W$.
We note that our notation here is somewhat nonstandard: we use $A\wr B$ and $A\Wr B$ to denote two
different permutations groups  (sets of permutations) obtained from $A$ and $B$, corresponding to two different actions of the abstract wreath product.

A standard way to visualize $A\Wr B$ is to think about elements being permuted as functions from $V^W$  (or, in other words, \emph{global sections} containing one point from each fibre) (see Figure~2; cf. Cameron~\cite[p. 103]{cam}). In case, when $|W|=n$ is finite, one can think about $n$-tuples of elements of $V$, with coordinates permuted by $\beta$, and at each coordinate permuted independently by permutations $\alpha_w$ (permutations of axes and independent permutations along axes).

It should be clear, that 
formula~$(\ref{eq:vw})$ above defines, in general, a one-to-one mapping from the set of pairs $(\beta, (\alpha_w)_{w\in W})$, where $\beta\in S_W$ and $(\alpha_w)_{w\in W}$ is a family of permutations $a_w\in S_V$, to the set of permutations $\phi$ of $V^W$. This follows easily from the well-known fact that the product action of the wreath product $S_V\Wr S_W$ is faithful. 
We will need a simple consequence of this 

\begin{Lemma}\label{l:fnotin}
If $\beta$ is a permutation of $W$ not in $B$, or one of $\alpha_w$ is a permutation of $V$ not in $A$, then $\phi$ given by $(\ref{eq:vw})$ does not belong to $A\Wr B$.
\end{Lemma}

The graph $G=G^*(A\Wr B)$ is much more complex than $G^*(A\wr B)$. 
We make only a few remarks. First, we may partition the edges of $G$ into $|W|$ classes. For each cardinal $1 \leq i \leq |W|$, we define the class $R_i$ as consisting of those edges $(f,g)$ of $G$ in which the values of the functions $f$ and $g$ differ at exactly $i$ points. It is clear, that any permutation $\phi \in A\Wr B$ preserves the partition into the classes $R_i$. The class $R_1$ may be viewed as the edges parallel to one of the axes of the coordinate system. In turn, $R_2$ may be viewed as edges contained in planes, not on lines, etc. 
One may observe that all lines parallel to one of the axes in $G=G^*(A\Wr B)$ contains copies of $G^*(A)$ whose colors depend on orbits of $B$. Generally, colors of other edges depend on actions of $B$ on functions whose values are orbits and orbitals of $A$. 

Even in case of transitive groups $A$ and $B$ the construction is complex and there is  no counterpart of Theorem~\ref{th:1}. Yet, using just the properties of the Galois connection we may prove, at least, the following useful fact.

\begin{Lemma}\label{l:awrb}
$$\clo{A\Wr B} \subseteq \clo{A}\Wr \clo{B}.$$  
\end{Lemma}
\begin{proof} We use the properties of the Galois connection.
We need to prove that $Aut(G^*(A\Wr B)) \subseteq \clo{A}\Wr \clo{B}$. This holds if and only if $G^*(Aut(G^*(A\Wr B))) \preceq G^*(\clo{A}\Wr \clo{B})$. The first term is equal to $G^*(A\Wr B)$. Thus, we need to prove that $G^*(A\Wr B) \preceq G^*(\clo{A}\Wr \clo{B})$. The latter follows by the obvious fact that $A\Wr B \subseteq \clo{A}\Wr \clo{B}$.
\end{proof}

\begin{figure}\label{fig2}
\begin{tikzpicture}[auto,inner sep=1pt, 
minimum size=3pt]

  
\draw[-] (1,0)--(7,0);
\draw[-] (1,0)--(1,4);
\draw[-] (2,0)--(2,4);
\draw[-] (3,0)--(3,4);
\draw[-] (4,0)--(4,4);
\draw[-] (5,0)--(5,4);
\draw[-] (6,0)--(6,4);
\draw[-] (7,0)--(7,4);
            
\draw[-,thick] (1,1.5)--(2,2) ; 
\draw[-,thick] (2,2)--(3,1.5) ; 
\draw[-] (3,1.5)--(4,2);
\draw[-] (3,1.5)--(4,1);

\draw[-] (4,2)--(5,3) ;
\draw[-] (4,1)--(5,1.5) ;

\draw[-] (5,3)--(6,2) ;
\draw[-] (5,1.5)--(6,2) ; 
\draw[-,thick] (6,2)--(7,1.5) ;

\draw (1,1.5) node[circle,fill=black,draw]{}; \draw (2,2) node[circle,fill=black,draw]{}; \draw (3,1.5)
node[circle,fill=black,draw]{};

\draw (4,2) node[circle,fill=black,draw]{};
\draw (4,1) node[circle,fill=black,draw]{};\draw (5,3) node[circle,fill=black,draw]{};\draw (5,1.5) node[circle,fill=black,draw]{}; 
\draw (6,2) node[circle,fill=black,draw]{}; \draw (7,1.5) node[circle,fill=black,draw]{};

\draw (0.7,3.5) node {$V$};
\draw (4,-0.5) node {$W$};

\draw (4.3,2.7) node {$f$};
\draw (5.5,1.5) node {$g$};
\draw (7.5,3.5) node {\ };

  \end{tikzpicture}
\caption{An edge  in $G^*(A\Wr B)$ of the class $R_2$ represented by a~pair of functions $\{f,g\}$ differing at exactly 2 points}
\end{figure}
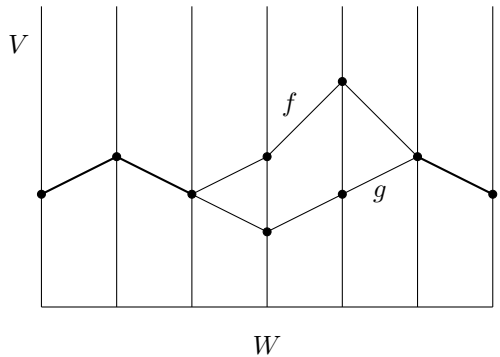

In \cite{je}, the following is proved. 

\begin{Theorem}\label{t:GR}
If $A, B \in GR$, then $A \Wr B \in GR$. 
\end{Theorem}

The proof in \cite{je} is only for finite permutation groups, but the part concerning the result formulated above generalizes easily to infinite groups.

Similarly as in the previous section we have the following

\begin{Lemma}\label{PNGR}
If  $A \notin DGR$, then $A \Wr B \notin GR$. 
\end{Lemma}

\begin{proof} Let $G=G^*(A \Wr B)$. Similarly as in Lemma~\ref{B}, we need to show that there exists $\phi\in Aut(G)\setminus (A \Wr B)$.   
Since $A\notin DGR$, there exists a permutation 
$\alpha\in \cld{A}\setminus A$.  
We fix an element $w_0\in W$ and define $f$ as follows
$$ (f\phi)(w) = \left\{ 
\begin{array}{ll}
(f(w))\alpha & \textrm{if } w = w_0,  \\ 
f(w) &  \textrm{otherwise, } 
\end{array} 
\right.$$ for any $w\in W$.
Thus $\phi$ acts almost as the identity modifying $f$ only for the value in $w=w_0$. 

First observe that, by Lemma~\ref{l:fnotin}, $\phi\notin A\Wr B$, as required. 
We proceed to show that $\phi$ preserves the colors of $G$.
Let $\{f,g\}$ be an edge of $G$. Then, 
$\{f,g\}\phi = \{f\phi,g\phi\}$, and we need to deal with the values $f(w_0)$ and $g(w_0)$.

We distinguish two cases. If
$f(w_0) \neq g(w_0)$,  then $(f(w_0),g(w_0))$ is a directed edge in $G(A)$. 
Since $\alpha$ preserves orbitals of $A$, this edge has the same color in $G(A)$ as $(f(w_0)\alpha,f(w_0)\alpha)$.

It follows that there exists $\alpha'\in A$ moving the directed  edge $(f(w_0),g(w_0))$ into $(f(w_0)\alpha,g(w_0)\alpha)$. Consequently, 
$ f(w_0)\alpha' = f(w_0)\alpha$ and $g(w_0)\alpha'= g(w_0)\alpha$. Now, the permutation $f'$ given by
$$ (f\phi')(w) = \left\{ 
\begin{array}{ll}
(f(w))\alpha' & \textrm{if } w = w_0,  \\ 
f(w) &  \textrm{otherwise, } 
\end{array} 
\right.$$ for any $w\in W$, belongs to $A\Wr B$ (since it is of the form (\ref{eq:vw})). We have $f\phi = f\phi'$ and  $g\phi = g\phi'$. Consequently, 
$\{f,g\}\phi = \{f,g\}\phi'$. The latter has the same color in $G^*$ as $\{f,g\}$, since $\phi'$ preserves the colors of edges. Thus, $\{f,g\}\phi$ has the same color as $\{f,g\}$, as required. 

For the second case, if $f(w_0) = g(w_0)$, then by Lemma~\ref{l:orbits},  there exists $\alpha'\in A$ such that $f(w_0)\alpha' = f(w_0)\alpha$, and in consequence, $g(w_0)\alpha'= g(w_0)\alpha$. The rest of the argument is the same as in the first case.
\end{proof}

Now, the situation is more complicated than in the case of imprimitive action of the wreath product. We need a new notion. Given a permutation group $(A,V)$, we will say that a permutation $\alpha$ of $V$ \emph{transposes orbitals} of $A$, if for every orbital $O$ of $A$, $O\alpha$ is the orbital paired with $O$. In such a case we say also that $A$ has \emph{transposable orbitals}. Obviously, such a permutation $\alpha$ belongs to $\clo{A}$. Moreover, 
if $A$ has at least one orbital that is not self-paired, then 
$\alpha\notin A$. 

Note also, that in case when all orbitals of $A$ are self-paired each $\alpha\in A$ transposes orbitals of $A$. In this case $A\in DGR$ implies $A\in GR$.

\begin{Lemma}\label{l:trans}
Let $A \in DGR \setminus (GR \cup \{I_2\})$.
If $A$ has transposable orbitals, then $A \Wr B \notin GR$.
\end{Lemma}

\begin{proof}
The proof is similar to the previous one. 
Let $G = G^*(A \Wr B)$. We need to show that there
exists $\phi \in Aut(G) \setminus A \Wr B$. Let $\alpha$ transposes orbitals of $A$. Then, since $A\notin GR$, $\alpha\notin A$.   We define $\phi$ as 
$(f\phi)(w) = (f(w))\alpha$, 
for any $w \in W$. 

First observe that, by Lemma~\ref{l:fnotin}, $\phi \notin A \Wr B$. Thus, it remains to show that $\phi$ preserves the colors of $G$. Let $\{f,g\}$ be an edge of $G$.
Since $\alpha$ transposes orbitals of $A$, for every $w\in W$, the directed edge  $(f(w),g(w))\alpha$ is in the same orbital of $A$ as $(g(w),f(w))$ (under assumption that $f(w) \neq g(w)$). Consequently, 
$(f(w)\alpha,g(w)\alpha)=(g(w)\alpha_w, f(w)\alpha_w)$ for some $\alpha_w \in A$.
If $f(w) = g(w)$, then by Lemma~\ref{l:orbits}, $f(w)\alpha=f(w)\alpha_w$ for some $\alpha_w \in A$.

Hence, for the permutation $\phi'$ given by $(f(w))\phi'=f(w)\alpha_w$, for any $w \in W$, we have $\{f,g\}\phi = \{f\phi, g\phi\} = \{g\phi', f\phi'\} = \{f,g\}\phi'$.   
The latter has the same color in $G$ as $\{f, g\}$, since $\phi'\in A \Wr B$.
Thus, $\{f,g\}\phi$ has the same color as $\{f, g\}$, which completes the proof.
\end{proof} 

Let us observe that the conclusion of the lemma does not hold (in general) for $A\in GR\cup\{I_2\}$,  since we know, by Theorem~\ref{t:GR}, that if $B\in GR$, then $A \Wr B \in GR$.  Similar situation is for $A=I_2$. Note that $I_2 \in DGR$ and it has transposable orbitals, but as we will see, sometimes $I_2\Wr B \in GR$, and sometimes not.

Thus, it remains to consider the cases when  $A\in DGR$ and either no $\alpha$ transposes orbitals of~$A$ or $A\in GR$ or $A=I_2$. We proceed to show that in all these cases $\clo{A\Wr B}= A\Wr B'$ for some $B'$. The case $A=I_2$ requires a separate proof.

\begin{Lemma}\label{l:i2}
For any permutation group $B$ $$\overline{I_2\Wr B} \subseteq I_2\Wr \overline{B}.$$
In particular, if $B\in GR$, then $I_2\Wr B\in GR$ .
\end{Lemma}

\begin{proof}
The elements of $I_2\Wr B$ are functions from $W$ to a $2$-element set $V$, say $V=\{0,1\}$. Let $R$ be the class of the edges $\{f,g\}$ in $G=G^*(I_2\Wr B)$, for which $f(w)=g(w) = 0$ except for one point $w$, in which $f(w)\neq g(w)$. We note that $R$ is invariant with regard to elements of $I_2\Wr B$. This means that, in $G^*(I_2\Wr B)$, the edges in $R$ have a color different from that of any edge not in $R$. By Lemma~\ref{l:awrb}, the elements $\phi$ of $Aut(G)$ have the form (\ref{eq:vw}). 
If any $\alpha_w$ in the definition of $\phi$ was a transposition, then for any element of $R$ with $f(w)=g(w)=0$ we would have $(f\phi)(w)=(g\phi)(w)=1$, which means that $\phi$ does not preserve the colors of edges, and contradict the fact $\phi\in Aut(G)$. Hence, $\clo{I_2\Wr B} \subseteq I_2 \Wr \clo{B}$, as required.  

For the second statement we use the fact that $B\in GR$ if and only if $B=\clo{B}$.
\end{proof}

We generalize the above lemma for the whole class in question.

\begin{Lemma} \label{l:nto}
Let $A \in DGR$. If no permutation $\alpha$ transposes orbitals of $A$ or $A\in GR\cup\{I_2\}$, then $\clo{A \Wr B} \subseteq A \Wr \clo{B}$. In particular, $\clo{A \Wr B} = A \Wr B'$ for some permutation group $B'$ with $B \subseteq B' \subseteq \clo{B}$.
\end{Lemma}

\begin{proof}
We prove the first claim.
For $A=I_2$ the claim is already proved in the lemma above. If $G\in GR$, then the claim is immediate by Lemma~\ref{l:awrb}. So we may assume that no permutation $\alpha$ transposes orbitals of $A$.

Let $G=G^*(A \Wr S_W)$.
In view of Lemma~\ref{l:awrb}, it is enough to show that $Aut(G) \subseteq A \Wr S_W$.
By the same lemma, we know that the elements $\phi\in Aut(G)$ are of the form (\ref{eq:vw}). Assume to the contrary that for some $\phi\in Aut(G)$, we have $\alpha_{w_1}\notin A$ for some $w_1\in W$. Again, by Lemma~\ref{l:awrb}, we have that $\alpha_{w_1}\in \clo{A}$.

We may assume also that $\beta$ in $\phi$ is the identity, that is, $\beta=id$. This is so, since any permutation of the form (\ref{eq:vw}) with all $\alpha_w$ equal to the identity is in $A\Wr S_W$, and thus in $Aut(G)$, and one may compose $\phi$ with a suitable permutation of this form to get $id$ in place of~$\beta$. 

Denote $\alpha_1=\alpha_{w_1}$. Since $\alpha_{1}\in \clo{A}\setminus A$,  there exist paired orbitals $O, O'$ of $A$, $O\neq O'$, and $(x,y) \in O$ such that  $(x\alpha_{1},y\alpha_{1})\in O'$. Since $(y,x)\in O'$ there exists $\alpha'\in A$ such that 
 $(x\alpha_{1}\alpha',y\alpha_{1}\alpha') = (y,x)$.  Replacing $\alpha_{1}$ by $\alpha_{1}\alpha'$ we may assume that  $(x\alpha_{1},y\alpha_{1}) = (y,x)$. Then we still have $\phi\in Aut(G)$ and $\alpha_{w_1}=\alpha_1\notin A$.
 
 In turn, fix $w_2\in W$ such that $w_2\neq w_1$. Denote $\alpha_2=\alpha_{w_2}$. Since $\alpha_{2}\in Aut(G)$,  there exist paired orbitals $Q, Q'$ of $A$, $Q\neq Q'$, and $(s,t) \in Q$ such that  $(s\alpha_{2},t\alpha_{2})\in Q$. {Otherwise, $\alpha_{2}$ transposes orbitals of $A$, contrary to the assumption.} Now,  since $(s,t)\in Q$ there exists $\alpha''\in A$ such that 
 $(s\alpha_{2}\alpha'',t\alpha_{2}\alpha'') = (s,t)$.  Again, replacing $\alpha_{2}$ by $\alpha_{2}\alpha''$ we may assume that  $(s\alpha_{2},t\alpha_{2}) = (s,t)$. The $\phi$ modified in such a way still belongs to $Aut(G)$.

\begin{figure}\label{fig3}
\begin{tikzpicture}[auto,inner sep=1pt, 
minimum size=3pt, xscale=1.4]

  
\draw[-] (1,0)--(7,0);

\draw[-] (1,0)--(1,3.5);
\draw[-] (2,0)--(2,3.5);
\draw[-] (3,0)--(3,3.5);
\draw[-] (4,0)--(4,3.5);
\draw[-] (5,0)--(5,3.5);
\draw[-] (6,0)--(6,3.5);
\draw[-] (7,0)--(7,3.5);

\draw[-] (2,2)--(3,2.2) ; 
\draw[-] (2,1)--(3,1.4) ; 

\draw[-, dashed,semithick] (2,2)--(3,1.4) ; 
\draw[-, dashed,semithick] (2,1)--(3,2.2) ; 

\draw[-] (3,2.2)--(4,1);
\draw[-] (3,1.4)--(4,1);

\draw[-,dashed, semithick] (3,2.2)--(4,2.5);
\draw[-, dashed, semithick] (3,1.4)--(4,2.5);

\draw[-,thick] (4,1)--(5,1) ;
\draw[-,thick] (5,1)--(6,1) ;
\draw[-,thick] (6,1)--(7,1) ;

\draw (2,2) node[circle,fill=black,draw]{};
\draw (2,1) node[circle,fill=black,draw]{};
 
\draw (3,2.2) node[circle,fill=black,draw]{};
\draw (3,1.4) node[circle,fill=black,draw]{};

\draw (4,2.5) node[circle,fill=black,draw]{};
\draw (4,1) node[circle,fill=black,draw]{};\draw (5,3) node[circle,fill=black,draw]{};\draw (5,1) node[circle,fill=black,draw]{}; \draw[-, dashed, thick] (4,2.5)--(5,3);

\draw (6,2) node[circle,fill=black,draw]{};
\draw (6,1) node[circle,fill=black,draw]{};\draw (7,1.5) node[circle,fill=black,draw]{};
\draw (7,1) node[circle,fill=black,draw]{};
 \draw[-, dashed, thick] (5,3)--(6,2);
 \draw[-, dashed, thick] (6,2)--(7,1.5);

\draw (0.7,3.5) node {$V$};

\draw (3.5,2.6) node {$f'$};
\draw (3.8,2) node {$g'$};
\draw (3.84,1.5) node {$f$};
\draw (3.5,1) node {$g$};
\draw (7.3,0) node {$W$};

\draw (1.8,2) node {$x$};
\draw (1.8,1) node {$y$};
\draw (2.86,2.35) node {$s$};
\draw (2.88,1.14) node {$t$};

\draw (2,-0.4) node {$w_1$};
\draw (3,-0.4) node {$w_2$};
  
  \end{tikzpicture}
\caption{The edge $\{f,g\}$ in the proof of Lemma~\ref{l:nto} (solid lines) and its image $\{f,g\}\phi = \{f',g'\}$ (dashed lines).}
\end{figure}
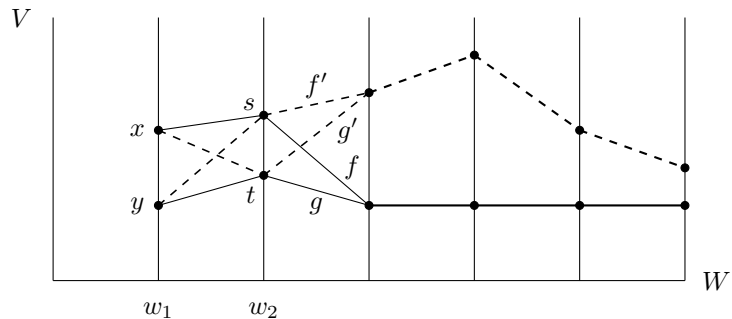

Now, we define two edges in $G$. For the first edge $\{f,g\}$ we put  $f(w_1) = x$, $f(w_2)=s$, $g(w_1)=y$, $g(w_2)=t$ and $f(w) = g(w) = v_0$, otherwise, for some fixed $v_0\in V$. For the second edge $\{f',g'\}$ we put  $f'(w_1) = x$, $f'(w_2)=t$, $g'(w_1)=y$, $g'(w_2)=s$ and $f'(w) = g'(w) = v_0\phi$, otherwise (see Figure~3). This and properties of $\alpha_{1}$ and $\alpha_{2}$ guarantee that $\{f,g\}\phi = \{f',g'\}$. Since $\phi\in Aut(G)$, the edges have the same color. 

Consequently, there exists a permutation $\psi\in A\Wr S_W$ such that $\{f,g\}\psi = \{f',g'\}$. We show that this leads to a contradiction. Let $(\beta', (\alpha'_w)_{w\in W})$ be the pair defining $\psi$ as in (\ref{eq:vw}).  Then we have either $f\psi = f'$ and $g\psi = g'$ or $f\psi = g'$ and $g\psi = f'$. In the first case we have $(x\alpha_1,s\alpha_2) = (x,t)$ and $(y\alpha_1,t\alpha_2) = (y,s)$. This implies $(s,t)\alpha_2 = (t,s)$, which contradicts the fact that $Q$ is not self-paired. In the second case we have $(x\alpha_1,s\alpha_2) = (y,s)$ and $(y\alpha_1,t\alpha_2) = (x,t)$. This implies $(x,y)\alpha_1 = (y,x)$, which contradicts the fact that $O$ is not self-paired. 
This completes the proof.

This proves the first statement. Combining it with Lemma~\ref{l:awrb},  we obtain $Aut(G) \subseteq A \Wr \clo{B}$, and since the set of $\beta$ such that there exists $\phi \in Aut(G)$ determined by (\ref{eq:vw}) forms a subgroup $B'$ of $\clo{B}$, we obtain $Aut(G) \subseteq (A \Wr B')$. 

Now, for any $\beta'\in B'$, there exists  $\phi\in Aut(G)$  of the form 
(\ref{eq:vw}) with $\beta=\beta'$ and  $\alpha_w\in A$ (since $Aut(G) \subseteq A \Wr \clo{B}$). It follows that the permutation $\phi'$ 
with $\beta=\beta'$ and all $\alpha_w=id$ is in $Aut(G)$ (since $\phi''$ with $\beta'$ replaced by $id$ is in $Aut(G)$). Consequently, any permutation of the form (\ref{eq:vw}) with all $\alpha_w\in A$ and $\beta\in B'$ belongs to $Aut(G)$. This implies that $\clo{A \Wr B} = A \Wr B'$, proving the second statement.
\end{proof}


A simple consequence of the lemma above, is that if in addition $B\in GR$, then $A \Wr B \in GR$. We show that, using this lemma, $GR$ in the condition $B\in GR$ may be replaced by a larger class $BGR$  of the automorphism groups of colored hypergraphs.

Formally, by a \emph{colored hypergraph} $H$ we mean a pair $H = (V,E)$, where $V$ is the set of vertices of $H$, with $|V|>1$, and $E$ the \emph{color function} from the set $P(V)$ of nonempty subsets of $V$ into the set of colors $\{ 0, \ldots, k-1\}$. 
An automorphism of a colored hypergraph $H=(V,E)$ is a permutation $\sigma$ of the set $V$
such that its extension on $P(V)$ preserves the colors of subsets. The automorphisms of $H$ form a permutation group on $V$, which is denoted by $Aut(H)$. In case of finite $V$, it may be visualized as the automorphism group of a colored symplex.

This group can be expressed also in terms of the  the invariant groups of systems of unordered relations (cf.~\cite{SV}) or the symmetry groups of $k$-valued boolean functions (cf.~\cite{kis1}). Our notation $BGR$ follows  \cite{GK1} and refers to the term ,,Boolean'' (in~\cite{SV} they are called \emph{orbit closed} groups).

The lemma below interestingly connects the classes $BGR$ and $GR$, studied separately so far.

\begin{Lemma} \label{BGR}
Let $A \in DGR$. If there is no permutation transposing orbitals of $A$ or $A\in GR\cup\{I_2\}$, then   
$B \in BGR$ implies that $A \Wr B \in GR$.  
\end{Lemma}

\begin{proof}
We need to show that  $\clo{A \Wr B} = A \Wr B$. 
By Lemma~\ref{l:nto},  $\clo{A \Wr B} = A \Wr B'$ with $B\subseteq B'$. Thus we need to show that $B'\setminus B$ is empty. 
Assume to the contrary that there exists $\gamma\in B'\setminus B$.

Since $B\in BGR$, there exists a colored hypergraph $H$ with $Aut(H)= B$. Since $\gamma\notin B$, there are two nonempty sets  $X, Y \subseteq V$ whose colors in $H$ are different, and $X\gamma = Y$. 

We define three functions $f,f_X$ and $f_Y$ belonging to $V^W$. Let $v_0\neq v_1$ be two fixed vertices of $V$. We define $f$ by $f(w)=v_0$ for all $w\in W$. For $f_X$ we put $f_X(w)=v_0$ if $w\in X$, and $f_X(w)=v_1$, otherwise. And similarly, $f_Y(w)=v_0$ if $w\in Y$, and $f_Y(w)=v_1$, otherwise.

Consider the edge $e = \{f,f_X\}$ in $G^*(A \Wr B)$. We will use the permutation $\phi_\gamma$ given by (\ref{eq:vw}) with $\beta=\gamma$ and $\alpha_w = id$ for all $w\in W$.
Since $\gamma\in B'$, $\phi_\gamma\in \clo{A \Wr B}$. Consequently, the edge $e\phi_\gamma$ has the same color  in $G^*(A \Wr B)$ as $e$. Therefore, there exist a permutation $\phi\in A\Wr B$ of the form (\ref{eq:vw}) such that $e\phi_\gamma = e\phi$.

Note that, by definition, $e\phi_\gamma = \{f,f_X\}\phi_\gamma = \{f,f_Y\}$. Hence, either $f\phi = f$ and $f_X\phi = f_Y$ or $f\phi = f_Y$ and $f_X\phi = f$. 

In the first case, it follows (from the first equation) that $w_0\alpha_w=w_0$ for all $w\in W$, and then (from the second equation) that $X\beta=Y$. This contradicts the fact that  $X$ and $Y$ have different colors in $H$ (since $\beta\in Aut(H)$). The second case is possible only in the case when $X\beta = Y$, which gives the same contradiction and completes the proof.  
\end{proof}

There is at least one case when the condition $B\in BGR $ in the lemma above is necessary.

\begin{Proposition}\label{p:sv}
The product $S_V \Wr B \in GR$ if and only if $B \in BGR$.
\end{Proposition}

\begin{proof}
The ``if'' part is proved in Lemma~\ref{BGR}. For the ``only if'' part, we assume, to the contrary that $B\notin BGR$ and $\clo{S_V\Wr B} = S_V \Wr B$. 

Given a subset $X\subseteq W$, let $R_X$ be the class of all edges $\{f,g\}$ in $G^*(S_V\Wr B)$ such that $f$ and $g$ differ at exactly the points of $X$. Obviously, all edges in this class are in the same $2^*$-orbital of $S_V\Wr B$. In fact, the 
$2^*$-orbitals of $S_V\Wr B$ are the unions of classes $R_X$ such that $R_X$ and $R_Y$ are in the same $2^*$-orbital if and only if $Y=X\beta$ for some $\beta\in B$. Consequently, any permutation $\phi$ of the form (\ref{eq:vw}) with $\beta$ preserving the orbits of $B$ on the subsets of $W$ belongs to $\clo{S_V\Wr B}$. 
Since $B \notin BGR$, there is $\beta \in S_V \setminus B$ that preserves all orbits of $B$ in this action. Hence, the corresponding $\phi \in  \clo{S_V \Wr B}$, but $\phi\notin S_V \Wr B$, which contradicts our assumption.  
\end{proof}

In general, a sufficient condition for $A\Wr B$ to belong to $GR$ is for $A$ to have a large enough the number of orbitals comparing with the cardinality of $W$. Let us recall that for a transitive group $A$ the rank of $A$, denoted $rank(A)$, is the number of orbits in the stabilizer of any point in $A$, which is equal to the number of orbitals. The latter allows to generalize this notion to intransitive groups. Thus, in general, by $rank(A)$ we denote the number of orbitals of $A$. Note, that this is equal to the number of nontrivial orbitals plus the number orbits of $A$.

\begin{Lemma}\label{l:oo}
Let $B \subseteq S_n$, $A \in DGR$, and assume that either no permutation transposes orbitals of $A$ or $A\in GR\cup\{I_2\}$.  If  $rank(A) \geq  n+1$ or $rank(A)=n$ and all orbitals of $A$ are self-paired, then $A \Wr B \in GR$.

\end{Lemma}

\begin{proof}
By Lemma~\ref{l:nto}, $\clo{A \Wr B} = A \Wr B'$ for some $B' \supseteq B$. We need to show that $B'=B$. 

Our general idea is to define a set $E$ of edges $e=\{f,g\}$ in $G=G^*(A \Wr B)$, which is regular (in a sense of \cite{SV}) with regard to action of $S_n$ on $W$. This would imply that the only permutations in $S_n$ preserving the colors of the edges in $E$ are those corresponding to action of $B$.

Later in this proof we use notation  $\beta = \phi[W]$ to specify the permutation of $W$ defining the permutation $\phi$ of $V^W$ by the formula (\ref{eq:vw}).

Let $E$ be the set of edges $e=\{f,g\}$ in in $G=G^*(A \Wr B)$ satisfying the following two conditions:

\begin{enumerate}
\item[(a)] there are $k$ points $w_i\in W$ at which $f(w_i)=g(w_i)$, and 
if $i\neq j$, then $f(w_i)$ and $f(w_j)$ are in different orbits of $A$,

\item[(b)] there are $n-k$ points $u_i\in W$ at which $f(u_i)\neq g(u_i)$, and if 
$i\neq j$, then $(f(u_i),g(u_i))$ and $(f(u_j),g(u_j))$ are in different (nontrivial) orbitals of $A$. 
\end{enumerate}
In case when $k\geq n$ we assume that only conditions (a) holds and only for $n$ points. Since  $rank(A) \geq n$, the set $E$ is nonempty. 

A crucial property is that all $f(w_i)$ in (a) are in different orbits of $A$, and all pairs $(f(u_i),g(u_i))$ in (b) are in different orbitals of $A$. By this property, it follows that if $\phi \in A \Wr B$  and $e\in E$, then $e\phi \in E$. We establish the properties of those permutations $\phi\in A \Wr S_n$ that fix edges in $E$.

Let $e\in E$ and $e=e\phi$ for some $\phi\in A \Wr S_n$.
As $\{f,g\}\phi = \{f,g\}$, we have either 

\emph{Case} (i): $f\phi = f$ and $g\phi=g$,  or 

\emph{Case} (ii): $f\phi = g$ and $g\phi=f$.

In Case (i), by the crucial property mentioned above, $\beta$ must fix all $w\in W$, that is $\beta=id$ is the identity. 

In Case (ii), $\beta$ fixes all $w_i$ of (a) and all $u_j$ of (b) with the property that $(f(u_i),g(u_i))$ belongs to a self-paired orbital (then $\alpha_{u_i}$ in $\phi$ transposes $f(u_i)$ and $g(u_i)$).
If $(f(u_i),g(u_i))$ is in an orbital $O$ of $A$ that is not self-paired,  then $\beta$ must transpose $u_i$ and $u_j$, where 
$(f(u_j),g(u_j))$ is in the orbital of $A$ paired with $O$. This case is possible only if for every $O$ represented by a pair $(f(u_i),g(u_i))$ in (b), also the orbital $O'$ paired with $O$ is represented in (b) by a pair $(f(u_j),g(u_j))$. Then, $\beta$ is the product of transpositions of pairs $(u_i,u_j)$ corresponding to pairs of non-self-paired orbitals represented in (b). If all orbitals of $A$ are self-paired, we have again, as in Case~(i), that $\beta$ must be the identity.

In case when $rank(A)\geq n+1$ and a non-self-paired orbital $O$ is represented in (b) by a pair $(f(u_i),g(u_i))$, we may assume that we choose points $u_i$ in $E$ in such a way, that there is no pair $(f(u_j),g(u_j))$ in (b) belonging to the orbital $O'$ paired with $O$. Then, Case (ii) is excluded, and it follows that, in any case we consider, $e=e\phi$ implies that $\beta=id$. 

Now, let $e=\{f,g\}$ be a fixed edge of the set $E$ as specified above, and $\phi$ be an arbitrary permutation in $\clo{A\Wr B} = A \Wr B'$. Thus, $\beta=\phi[W]\in B'$. 
All we need to show is that $\beta\in B$. By the assumption, $\phi$ preserves the colors of the edges in $G=G^*(A\Wr B)$. Hence, there exists $\phi'\in A \Wr B$ such that $e\phi = e\phi'$. Let $\beta'= \phi'[W]$. We infer that $e = e\phi'\phi^{-1}$. Since $(\phi'\phi^{-1})[W] = \beta'\beta^{-1}$, by what we have established above, $\beta'\beta^{-1} = id$, and consequently, $\beta=\beta'\in B$, as required.
\end{proof}

We note that the proof holds for arbitrary cardinals $rank(A)$ and $n$, as well.
On the other hand, if $n$ is large enough comparing with $rank(A)$, then  $A\Wr B$ may not belong to $GR$. We settle the case when $B=A_n$ is the alternating group on $n$ elements. This depends on the number of pairs of non-self-paired orbitals in $A$. Denote this number by $nsp(A)$.

\begin{Lemma}\label{l:aaa} Let $A \in DGR$, and  either no permutation transposes orbitals of $A$ or $A\in GR\cup\{I_2\}$. 
If  $rank(A)=n$ and $nsp(A)$ is even, then $A \Wr A_n  \in GR$. 
\end{Lemma}

\begin{proof}
We continue the proof of the previous lemma with $B=A_n$. The only difference is that now Case~(ii) is not excluded, and in last line, where we consider the product $\beta'\beta^{-1}$, we need to take into consideration that, according to in Case~(ii),  $\beta'\beta^{-1} =\tau$ is the product of transpositions of pairs $(u_i,u_j)$ corresponding to pairs of non-self-paired orbitals represented in (b). Then we have 
$\beta=\beta'\tau^{-1}$, and since $nsf(A)$ is even, $\tau\in A_n$, and consequently, $\beta\in A_n$, as required.
\end{proof}

\begin{Lemma}\label{l:aaa2}
If $rank(A) < n$ or $rank(A) =n$ and    $nsp(A)$ is odd, then  $A \Wr A_n  \notin GR$. 
\end{Lemma}

\begin{proof}
It is enough to show that $\clo{A \Wr A_n} \supseteq A \Wr S_n$. To this end, for any edge $e=\{f,g\}$ in $G^*(A \Wr A_n)$ and any permutation $\phi\in A \Wr S_n$ we need to  show that there exists a permutation $\phi'\in A \Wr A_n$ such that if $e\phi = e\phi'$. Let $\phi$ be of the form (\ref{eq:vw}). If $\beta\in A_n$, we are done. So, assume that $\beta\in S_n\setminus A_n$.

Consider the first case with 
$n>rank(A)$. Then, there exists $w_1,w_2\in W$ such that either $f(w_1)=g(w_1)$ is in the same orbit as $f(w_2)=g(w_2)$ or 
$(f(w_1),g(w_1))$  and  $(f(w_2),g(w_2))$ are in the same orbital. Let $\alpha\in A$ be such that $f(w_1)\alpha = f(w_2)$, in the first case, or $(f(w_1),g(w_1))\alpha =(f(w_2),g(w_2))$, in the second case. 

Let $\tau=(w_1,w_2)$ denote the transposition of $w_1$ and $w_2$. Define $\beta'=\tau\beta$. Then, clearly, $\beta'\in A_n$. 
Identify $\beta'$ with the permutation of $V^W$ acting only as permuting fibres, and
let $\alpha$ be identified with the permutation of $V^W$ with $\alpha$ acting solely on the fiber $w_2$. 
Define $\phi'$ as $\alpha\beta'\alpha^{-1}\phi$. Then, obviously, $f\phi = f\phi'$ and $g\phi = g\phi'$
and consequently, $e\phi = e\phi'$, as required.

In the second case, $rank(A)=n$ and $nsp(A)$ is odd, we have either the situation as above (then the claim is proved) or the situation is as in the proof of Lemma~\ref{l:oo} with all $f(w_i)$ in different orbits and all pairs $(f(u_i),g(u_i))$ in different orbitals. Then, there exist a permutation $\tau$ being the product of transpositions of pairs $(u_i,u_j)$ corresponding to pairs of non-self-paired orbitals such that $e\phi_{\tau} = e$ (as before, by $\phi_{\tau}$ we denote the corresponding permutation of the form (\ref{eq:vw}) with $\beta=\tau$ and all $\alpha_w = id$). Since $nsp(A)$ is odd,   $\tau\notin A_n$. Consequently,  
since $\beta\in S_n\setminus A_n$, $\tau\beta \in A_n$. Hence, $\phi'= \phi_{\tau}\phi \in A\Wr A_n$. We have $e\phi'=e\phi_{\tau}\phi = e\phi$, which completes the proof.
\end{proof}

We summarize the results of this chapter. Let $DGR^+$ denotes the class of permutation groups $A\in DGR$ such that either $A$ has not transposable orbitals or $A\in GR \cup \{I_2\}$. The most complete result of this section is

\begin{Theorem}
If a permutation group $B\in BGR$, then for any permutation group $A$, the product $A\Wr B\in GR$ if and only if $A\in DGR^+$.
\end{Theorem}

This follows from Lemmas~\ref{PNGR}, \ref{l:trans} and \ref{BGR}.
In case of directed graphs the same proofs apply when considering only the cases with the assumption that no permutation transposes orbitals of the group in question. This yields the following. 

\begin{Theorem}
If a permutation group $B\in BGR$, then for any permutation group $A$, the product $A\Wr B\in DGR$ if and only if $A\in DGR$.
\end{Theorem}

For $B\notin BGR$ we have only the following partial result.

\begin{Theorem} \label{th:Wr}
Let $(A,V)$ and $(B,W)$ be  permutation groups. If  $(B,W)\notin BGR$, then the following hold:  
\begin{enumerate}
\item 
If $A\notin DGR^+$, then $A\Wr B \notin GR$;

\item
If $A\in DGR^+$, and either 
 $rank(A)\geq |W|+1$ or 
$rank(A)=|W|$ and all orbitals of $A$ are self-paired,
then $A\Wr B \in GR$.
\end{enumerate}
\end{Theorem}

In addition, we have the full characterization in the cases when $A=S_V$ (Proposition~\ref{p:sv}) or  $B=A_n$. By Lemmas~\ref{l:oo}, \ref{l:aaa} and \ref{l:aaa2} we get (with $nsp(A)$ denoting the number of pairs of non-self-paired orbitals of $A$):

\begin{Proposition}
The product $A\Wr A_n \in GR$ if and only if $A\in DGR^+$ and one of the following holds
\begin{enumerate}
    \item $rank(A)\geq n+1$, or
    \item $rank(A)=n$ and $nsp(A)$ is even.
\end{enumerate}
\end{Proposition}

The above results hold for ininite groups and infinite cardinals $n$, as well. In case of directed graphs the situation is now more complicated. Our proofs may be easily modified to obtain analogous results in case of Lemma~\ref{PNGR}, Lemma~\ref{l:nto}, and Proposition~\ref{p:sv}. For other results, the main problem is that the counterpart of Theorem~\ref{t:GR} does not hold. The counterexample is given in the next section.

\section{Concluding remarks and open problems}

The main open problem raised by this paper is, as we have already mentioned, to complete the characterization in Theorem~\ref{th:Wr}. In view of the consideration in the previous section, this problem seems rather hard and may lead to many technical conditions. A more approachable, and a good starting point seems the following.
\medskip

\textbf{Problem 1}. Characterize those permutation groups $B$ for which the product $I_k \Wr B \in GR$.
\medskip

This may be connected with the problem of a characterization of the class $BGR$. In fact, we have very little examples of permutation groups not in $BGR$, i.e., those that are not the automorphism groups of colored hypergraphs. 
The well-known examples are the alternating groups $A_n$ for $n\geq 3$ and $C_4$ and $C_5$ (see \cite{GK1}). A few further examples of exceptional character are primitive groups given in \cite{ser} (cf. \cite[Theorem~2.6]{SV}) (in the terminology applied therein they are not \emph{orbit closed}). Using these examples one can construct further examples via various product operations. Yet, no characterization of $BGR$ is known.

Although the group $C_3 \notin BGR$, it belongs to $DGR$, and this exceptional property is used to construct a counterexample mentioned in the previous section. One may check directly that the group $C_2\Wr C_3$ is neither in $DGR$ nor in $BGR$. In particular, we have that $C_2,C_3 \in DGR$, which shows that the counterpart of Theorem~\ref{t:GR} for $DGR$ does not hold. Generally, one can prove that $S_V\Wr C_m \notin DGR$ for any  $m=3,4,5$.

The groups $C_3, C_4, C_5$  are exceptional in that all they belong to $DGR\setminus BGR$. One may construct more such examples using, again, various product operations, but we do not know any other 
transitive example not involving $C_i$ for $i=3,4,5$.
\medskip

\textbf{Problem 2}. Is there any primitive permutation group $A\in DGR $,  other than  $C_3$ and $C_5$, that does not belong to $BGR$?
\medskip

It is a good place to comment on Section~5 in \cite{SV} entitled ``Wreath products''. We need to note that Theorem~5.1 (saying that if $A,B\in BGR$, then $A\wr B\in BGR$, and if $B\notin BGR$, then $A\wr B\notin BGR$) 
has been proved much earlier in \cite[Theorem~15]{CK}, and in more general form in \cite[Theorems~5.3 and 5.4]{kis1}. Moreover, Corollary~5.3 in this section, giving \emph{via} the wreath product an infinite number of examples of groups in $BGR$ that are not \emph{relation groups} (meaning: are not the automorphism groups of hypergraphs) is false. Given as an example the group $L\wr C_3$  is not orbit closed (here and above we use our notation, reverse to that in \cite{SV}). The method suggested in the proof of this corollary yields on other good example. So the problem below, posed in \cite{kis1} (after discovering that \cite[Theorem~13]{CK} is false) remains still open.

First, recall that the Klein $4$-group $K_4 \subseteq S_2\wr S_2$ can be easily seen to be the automorphism group of a 3-colored hypergraph, but it is also easy to see that it is not the automorphism group of any (uncolored) hypergraph (cf. \cite{kis1}).
\medskip

\noindent\textbf{Problem 2}. Is there any example of a permutation group $G\in BGR$, other than $G=K_4$, that is  not the automorphism group of any (uncolored) hypergraph? 
\medskip

\footnotesize


\end{document}